\documentclass[12pt]{amsart}
\usepackage{psfrag}
\usepackage{amsrefs}
\usepackage{wrapfig, subfigure,graphicx}
\usepackage{hyperref}
\usepackage{amssymb}
\usepackage{amsthm}
\usepackage{amsmath}
\newcommand\numberthis{\addtocounter{equation}{1}\tag{\theequation}}
\usepackage{graphicx}
\usepackage{fullpage}
\usepackage{color}

\numberwithin{equation}{section}
\newcommand{\esssup}{ess\,sup}

\newenvironment{pfn}{\noindent{\sc Proof}}{\rule{2mm}{2mm}\medskip}

\DeclareMathOperator*{\essup}{ess\,sup}

\theoremstyle{plain}
\newtheorem{theorem}{Theorem}[section]

\newtheorem{lemma}[theorem]{Lemma}
\newtheorem{proposition}[theorem]{Proposition}
\newtheorem{definition}[theorem]{Definition}
\newtheorem{remark}[theorem]{Remark}

\renewcommand{\theequation}{\thesection.\arabic{equation}}

\begin{document}
\title{Optimal control for the infinity obstacle problem }

\author[H. Mawi,  and C. B. Ndiaye]{Henok Mawi \; and \;Cheikh Birahim Ndiaye}
\thanks{\today. \\
The first author was partially supported by NSF grant HRD--1700236.}
\address{Department of Mathematics, Howard University, Washington, D.C. 20059}
\email{henok.mawi@howard.edu, cheikh.ndiaye@howard.edu}

\maketitle
\begin{abstract}
In this note, we show that a natural optimal control problem for the \;$\infty$-obstacle problem admits an optimal control which is also an optimal state. Moreover, we show the convergence of the minimal value of an optimal control problem for the \;$p$-obstacle problem to the minimal value of our optimal control problem for the \;$\infty$-obstacle problem, as \;$p\to\infty$.
\end{abstract}

\section{Introduction}
The obstacle problem corresponding to an obstacle \;$f$ in
\begin{equation}\label{11}
 W^{1,2}_g (\Omega) =  \{ u \in W^{1,2} (\Omega) : \;\;u = g \,  \;\;\;\textrm{on} \,  \;\;\;\partial \Omega\}
\end{equation}
consists of minimizing the Dirichlet energy
\[
\int_{\Omega} |Du(x)|^2 \, dx
\]
over the set 
\begin{equation}\label{eq:trace}
\mathbb K^2_{f, g} = \{ u \in W^{1,2}_g (\Omega) : \;\;u(x) \geq f(x) \,\; \;\;\textrm{in} \ \, \Omega\}
\end{equation}
where  \;$\Omega \subset \mathbb R^n$\; is a bounded and smooth domain, \;$Du$\;is the gradient of \;$u$, \; and \;$g \in tr(W^{1, 2}(\Omega))$ with \;$tr$\; the trace operator. In \eqref{11}, the equality \;$u= g \,\; \textrm{on} \,\; \partial \Omega$ \;is in the sense of trace. 
This problem is used to model the equilibrium position of an elastic membrane whose boundary is held fixed at $g$ and is forced to remain above a given obstacle $f.$ It is known that the obstacle problem admits a unique solution $v \in \mathbb{K}^2_{f, g} $. That is, there is a unique  \;$v \in \mathbb{K}^2_{f, g} $ such that
\[
\int_{\Omega} |Dv(x)|^2 \, dx \leq \int_{\Omega} |Du(x)|^2 \, dx, \;\;\;\;\forall u \in  \mathbb{K}^2_{f, g}.
\]
\vspace{8pt}

In \cite{ALY98} Adams, Lenhart and Yong introduced an optimal control problem for the obstacle problem by studying the minimizer of the functional 
\[
J_2(\psi) = \frac{1}{2} \int_{\Omega} (|T_2(\psi) - z|^2 + |D\psi|^2 )\, dx.
\]
In the above variational problem, following the terminology in control theory \cite{L71}, $\psi$ is called the control variable and $T_2(\psi)$ is the corresponding state. The control $\psi$ lies in the space  $W^{1,2}_0(\Omega)$, the state  $T_2(\psi)$ is the unique solution for the obstacle problem corresponding to the obstacle $\psi$ and the profile $z$ is in $L^2(\Omega).$  The authors proved that there exists a unique minimizer $\bar \psi \in W^{1,2}_0(\Omega)$ of the functional $J_2$. Furthermore, they showed that \;$T_2(\bar \psi) = \bar\psi.$ 
\vspace{8pt}

Following suit, for \;$1 < p < \infty$, and $z \in L^p(\Omega)$, Lou in \cite{LH02} considered the variational problem of minimizing the functional 
\begin{equation}\label{eq:pp}
 \tag{$P_{p}$}
\bar J_p(\psi) = \frac{1}{p} \int_{\Omega} |T_p(\psi) - z|^p + |D\psi|^p\, dx
\end{equation}
for 
$
\psi \in W^{1,p}_0(\Omega):=\{ u \in W^{1,p} (\Omega) : \;\;u=0 \;\;\; \textrm{on} \, \;\; \partial \Omega\}
$
and established that the problem admits a minimizer $\bar \psi.$  Here $T_p(\psi)$ is the unique solution for the $p-$obstacle problem with obstacle \;$\psi\in  W^{1,p}_0(\Omega)$, see \cite{ALS15} and references therein for discussions about the \;$p$-obstacle problem. We remind the reader that the $p-$obstacle problem with obstacle $f \in W^{1,p}_g (\Omega)$ refers to the problem of minimizing the $p-$Dirichlet energy
\[
\int_{\Omega} |Du(x)|^p \, dx 
\]
among all functions in the class
\begin{equation*}
\mathbb K^p_{f, g} = \{ u \in W^{1,p} (\Omega) : \;\;\;u \geq f \, \;\;\; \textrm{in} \ \,\;\; \Omega \, \;\;\; \textrm {and} \, \;\;u = g \, \;\;\; \textrm{on} \,  \;\;\;\partial \Omega\},
\end{equation*}
with \;$g\in tr(W^{1, p}(\Omega)$.
It is further shown in \cite{LH02} that, as in the case of $p=2,$ $T_p(\bar \psi) = \bar \psi.$ 
\vspace{8pt}

For the boundary data \;$g\in Lip(\partial \Omega)$, letting \;$p \to \infty$, one obtains a limiting variational problem of $L^{\infty}$-type which is referred in the literature as the infinity obstacle problem or $\infty$-obstacle problem (see \cite{RTU15}) . That is, given an obstacle \;$f \in W^{1, \infty}_g(\Omega)$ one considers the minimization problem: 
\begin{equation}\label{inftyobstacelvariational}
\text{Finding }\;\;u_{\infty}\in  \mathbb K^{\infty}_{f, g}:\;\;\;  ||Du_{\infty}||_\infty=\inf_{u \in \mathbb K^{\infty}_{f, g}}  ||Du||_\infty,
\end{equation}
where 
\begin{equation*}\label{eq:defkti}
\mathbb K^{\infty}_{f, g} = \{ u \in W^{1,\infty}(\Omega) : \;\;v \geq f\;\;\; \textrm{in} \;\,\;\; \Omega \;\;\;\, u = g \,  \;\;\;\textrm{on} \, \;\;\; \partial \Omega\}, \;\;\text{and} \;\;\;||\cdot||_{\infty}:=\esssup\;|\cdot |.
\end{equation*}
It is established in \cite{RTU15} that the minimization problem \eqref{inftyobstacelvariational} has a solution \;
\begin{equation}\label{c1}
u_{\infty}:=u_{\infty}(f)\in \mathbb K^{\infty}_{f, g}
\end{equation}
 which verififies
\begin{equation}\label{infinityp}
-\Delta_{\infty}u_{\infty}\geq 0 \;\;\; \text{in}\;\;\Omega\;\;\; \text{in a weak sense }.
\end{equation}
More importantly, the authors in \cite{RTU15} characterize \;$u_{\infty}$\; as the smallest infinity superharmonic function on $\Omega$ that is larger than the obstacle $f$ and equals $g$ on the boundary. Thus  for  a fixed \;$F\in Lip(\partial \Omega)$, this generates an obstacle to solution operator  $$T_{\infty}: W^{1,\infty}_F(\Omega)\longrightarrow W^{1,\infty}_F(\Omega)$$ defined by
\begin{equation}\label{eq:definitionti}
T_{\infty}(f):=u_{\infty}(f)\in W^{1,\infty}_F(\Omega),\;\;\;\;f \in W^{1,\infty}_F(\Omega),
\end{equation}
 where $$W^{1,\infty}_F(\Omega):= \{ u \in W^{1,\infty}(\Omega) : u = F\,  \;\;\;\textrm{on} \, \;\;\; \partial \Omega\}.$$

\vspace{10pt}

In this note, we consider a natural optimal control problem for the infinity obstacle problem. More precisely, for \;$F\in Lip (\partial \Omega)$\; and  for  \;$z \in L^{\infty}(\Omega)$\; fixed,  we introduce the functional
\begin{align*}
J_{\infty} (\psi) =  \max\{||T_{\infty}(\psi) - z||_{\infty}  , \, ||D\psi||_{\infty} \},\;\;\;\psi\in W^{1, \infty}_F(\Omega)
\end{align*}
and study the problem of existence of \;$\psi_{\infty} \in W^{1,\infty}_F (\Omega)$\; such that:

\begin{align}
\tag{$P_{\infty}$}
J_{\infty} (\psi_{\infty}) \leq  J_{\infty} (\psi) ,\qquad\;\forall  \quad \psi  \in  W^{1,\infty}_F (\Omega).
\label{eq:inftycase}
\end{align}
In deference to optimal control theory, a  function \;$\psi_{\infty}$\; satisfying \eqref{eq:inftycase} is called an \emph {optimal control} and the state \;$T_{\infty}(\psi_{\infty})$\; is called an\emph{ optimal state.}  
\vspace{10pt}

Several variants of control problems where the control variable is the obstacle have been studied by different authors since the first of such works appeared in \cite{ALY98}. The literature is vast, but to mention a few, in \cite{AL03} the authors studied a generalization of \cite{ALY98} by adding a source term. In \cite{AL02} a similar problem is studied when the state is a solution to a parabolic variational inequality. In \cite{L01} the author studied regularity of the optimal state obtained in \cite{ALY98}. When the state is governed by a bilateral variational inequality, results are obtained in \cite{BL04}, \cite {QC00},  \cite{CY05} and \cite{CCT05}. Optimal control for higher order obstacle problems appears in \cite{AHL10} and \cite{GN17}. Related works where the control variable is the obstacle are also studied in \cite{DM15, SM12} and the references therein.
\vspace{10pt}

In this note, we prove that the optimal control problem \eqref{eq:inftycase} associated to \;$J_{\infty}$\; is solvable. Precisely we show the following result:
\begin{theorem}\label{eq:main}
Assuming that \;$\Omega \subset \mathbb R^n$\; is a bounded and smooth domain, $F\in Lip (\partial \Omega )$, and  \;$z \in L^{\infty}(\Omega)$, \;$J_{\infty}$\; admits an optimal control \;$u_{\infty}\in W^{1, \infty}_F(\Omega)$\; which is also an optimal state, i.e\; $$u_{\infty}=T_{\infty}(u_{\infty}).$$
\end{theorem}
\vspace{10pt}

Using also arguments similar to the ones used in the proof of Theorem \ref{eq:main}, we show the convergence of the minimal value of an optimal control problem associated to \;$\bar J_p$\; to the minimal value of the optimal control problem corresponding to \;$J_{\infty}$\; as \;$p$ \;tends to infinity. Indeed we prove the following result:
\begin{theorem}\label{eq:convminima}
Let \;$\Omega \subset \mathbb R^n$\; be a bounded and smooth domain, $F\in Lip (\partial \Omega )$, and  \;$z \in L^{\infty}(\Omega)$. Then setting
\[
J_p= (p\bar J_p)^{\frac{1}{p}}, \;\;C_p = \min_{\psi \in W^{1, p}_F(\Omega)}J_p(\psi)\;\;\text{for}\;\;\;1<p<\infty,\;\;
\text{and}\;\;\;\;\;C_{\infty} = \min_{\psi \in W^{1,\infty}_F(\Omega)}J_{\infty}(\psi),
\]
where \;$\bar J_p$\; is as in \eqref{eq:pp}, we have \[
\lim_{p \to \infty}C_p =C_{\infty}
\]

\end{theorem}
\vspace{10pt}
In the proofs of the above results, we use the \;$p$-approximation technique as in the study of the \;$\infty$-obstacle problem combined with the classical methods of weak convergence in Calculus of Variations. As in the study of the \;$\infty$-obstacle problem, here also the key analytical ingredients are the \;$L^q$-characterization of \;$L^{\infty}$\; and H\"older's inequality. The difficulty arises from the  the fact that the unicity question for the $\infty$-obstacle problem is still an open problem to the best of our knowledge.   To overcome the latter issue, we make use of the characterization of the solution of the \;$\infty$-obstacle problem by Rossi-Teixeira-Urbano \cite{RTU15}.
\section {Preliminaries}
One of the  most popular way of approaching problems related to minimizing a functional of \;$L^{\infty}$-type is to follow the idea first introduced by Aronsson in \cite{GA65} and which involves interpreting an\;$L^{\infty}$-type minimization problem as a limit when \;$p \to \infty$\; of an \;$L^{p}$-type minimization problem. In this note, this \;$p$-approximation technique will be used to show existence of an optimal control for  \;$J_{\infty}$. In order to prepare for our use of the \;$p$-approximation technique, we are going to start this section by discussing some related \;$L^p$-type variational problems. 
\vspace{10pt}

Let \;$\Omega \subset \mathbb R^n$\; be a bounded and smooth domain and \;$g \in Lip(\partial\Omega).$  Moreover let \;$\psi \in W^{1,\infty}_g(\Omega)$\; be fixed  and \;$1<p<\infty$. Then as described earlier the \;$p$-obstacle problem with obstacle \;$\psi$\;corresponds to finding a minimizer of the functional
\begin{equation}\label{pobstaclevariationalform}
I_p(v) = \int_{\Omega} |Dv(x)|^p dx 
\end{equation}
over the space \;\;$\mathbb{K}^p_{\psi, g} = \{ v \in W^{1,p}(\Omega) : \;v \geq \psi, \;\;\; \textrm {and} \, \;\;\;v= g \,\;\;\;  \textrm{on} \,\;\;\;  \partial \Omega\}$.  The energy integral (\ref{pobstaclevariationalform}) admits a unique minimizer \;$u_p\; \in \;\mathbb{K}^p_{\psi, g} .$ The minimizer \;$u_p$\;is not only \;$p$-superharmonic, i.e \;$\Delta_p u_p \leq 0$,  but is also a weak solution to the following system
\begin{equation}\label{pobstacleproblem}
\begin{cases}
-\Delta_p u \geq  0  & \quad \quad \textrm{in} \quad \Omega\\
-\Delta_p u \, (u-\psi) =  0  & \quad \quad \textrm{in} \quad \Omega\\
u  \geq \psi &\quad \quad  \textrm{in} \quad \Omega
\end{cases}
\end{equation}
where $\Delta_p$ is the $p$-Laplace operator given by
\[\Delta_p u  : =div ( |Du|^{p-2} Du).\]
Moreover, it is known that the \;$p$-obstacle problem is equivalent to the system \eqref{pobstacleproblem} (see \cite{L71} or \cite{L15}) and hence we will refer to \eqref{pobstacleproblem} as the \;$p$-obstacle problem as well. On the other hand, by the equivalence of weak and viscosity solutions established in \cite{L15} (and \cite{JJ12} )\;$u_p$\;  is also a viscosity solution of \eqref{pobstacleproblem} according to the following definition.
\begin{definition} A function $u \in C(\Omega)$ is said to be a viscosity subsolution (supersolution) to 
\begin{equation}
\begin{aligned}\label{NLPDE}
F(x, u, Du, D^2 u) & = 0 \quad \textrm{in}  \, \quad \Omega\\
u & = 0  \quad \textrm{in}  \, \quad \partial \Omega
\end{aligned}
\end{equation}
if for every $\phi \in C^2(\Omega)$ and $x_0 \in \Omega$ whenever $\phi-u$ has a minimum  (resp. maximum) in a neighborhood of \;$x_0$\; in \;$\Omega$\; we have:
\[
F(x, u, D\phi, D^2 \phi)  \leq 0 \quad (\textrm{resp. } \quad \geq 0).
\] 
The function\; $u$\; is called a viscosity solution of \eqref{NLPDE} in \;$\Omega$\; if \;$u$\; is both viscosity subsolution and viscosity supersolution of \eqref{NLPDE} in \;$\Omega.$
\end{definition}
\vspace{10pt}

The asymptotic behavior of the sequence  of minimizers \;$(u_p)_{p>1}$\; as \;$p$\; tends to infinity has been investigated in \cite{RTU15}.  In fact, in \cite{RTU15}, it is established that  for a fixed \;$\psi\in W^{1, \infty}_g(\Omega)$, there exists \;$u_{\infty}=u_{\infty}(\psi)\in \mathbb{K}^{\infty}_{\psi, g} = \{ v \in W^{1,\infty}_g(\Omega) : v \geq \psi\}$ \;such that \;$u_p \to u_\infty$\; locally uniformly in $\bar\Omega$, and that for every \;$q\geq 1$, \; $u_p$\; converges to \;$u_{\infty}$\; weakly in \;$W^{1,q}(\Omega).$ Furthermore,  $u_\infty$\; is a solution to the \;$\infty$-obstacle problem
\begin{equation}\label{inftyobstacelvariationalform}
\min_{v \in \mathbb{K}^{\infty}_{\psi, g}}  ||Dv||_\infty
\end{equation}
For  \;$\Omega$  convex (see \cite{ACJ04}), the variational problem \eqref{inftyobstacelvariationalform} is equivalent to the minimization problem
\begin{equation*}\label{inftyobstacelvariationalform1}
\min_{v \in \mathbb{K}^{\infty}_{\psi, g}}  \mathcal{L}(v),
\end{equation*}
where
\[
\mathcal{L}(v) = \inf_{(x, y)\in \Omega^2, \;x \neq y}\dfrac{|v(x) - v(y)|}{|x - y|}.
\]
Moreover, in \cite{RTU15}, it is show that \; $u_\infty$\; is a viscosity solution to the following system.  
\begin{equation*}\label{obstacleproblemforinftylaplacian}
\begin{cases}
-\Delta_ \infty u \geq  0  &\quad \quad \textrm{in} \quad \Omega\\
-\Delta_ \infty u \, ( u-\psi)= 0  &\quad \quad \textrm{in} \quad \Omega\\
u  \geq \psi &\quad \quad \textrm{in} \quad \Omega
\end{cases}
\end{equation*}
where \;$\Delta_ \infty$\; is the \;$\infty$-Laplacian and is defined by 
\[
\Delta_\infty u  = \langle D^2uDu, Du\rangle =\sum_{i=1}^n\sum_{j=1}^nu_{x_i}u_{x_j} u_{x_ix_j}.
\] 
\vspace{10pt}

Recalling that \;$u$\; is  said to be { \it infinity superharmonic} or  \;$\infty$-{ \it superharmonic}, if $-\Delta_ \infty u \geq  0$\; in the viscosity sense, we have the following characterization of \;$u_{\infty}$\; in terms of infinity superharmonic functions  and it is proven in \cite{RTU15}. We would like to emphasize that this will play an important role in our arguments.
\begin{lemma}\label{eq:infinfhar}
Setting $$\mathcal{F}^+=\{v\in C(\Omega), \; \;-\Delta_ \infty v\geq 0\;\;\text{in}\;\;\Omega \;\;\text{in the viscosity sense}\}$$\; and \;$$\mathcal{F}^+_\psi=\{v\in \mathcal{F}^+,\;\;v\geq\psi \;\;\text{in}\;\;\Omega,\;\;\text{and}\;\;v=\psi\;\;\text{on}\;\;\partial \Omega\},$$ we have
\begin{equation}\label{eq:infinf}
T_{\infty}(\psi)=u_{\infty}=\inf_{v\in \mathcal{F}^+_\psi} v,
\end{equation}
with \;$T_{\infty}$\; as defined earlier in \eqref{eq:definitionti}.
\end{lemma}
\vspace{10pt}

Lemma \ref{eq:infinfhar} implies the following characterization of infinity superharmonic functions as fixed points of \;$T_{\infty}$.  This charactreization plays a key role in our \;$p$-approximation scheme for existence.
\vspace{8pt}



\begin{lemma}\label{eq:inftc}
Assuming that \;$u\in W^{1, \infty}_g(\Omega)$, \;$u$\; being infinity superharmonic is equivalent to  \;$u$\; being a fixed point of \;$T_{\infty}$, i.e \;$$T_{\infty}(u) =u.$$ 
\end{lemma}
\begin{proof}
Let \;$u\in W^{1, \infty}_g(\Omega)$ \; be an infinity superharmonic function and \;$v$\; be defined by \;$v=T_{\infty}(u)$. Then clearly the definition of \;$v$\; and lemma \ref{eq:infinfhar} imply \;$v \geq u$. On the other hand, since \;$u\in W^{1, \infty}_g(\Omega)$\; and is an  infinity superharmonic function,  we deduce from lemma \ref{eq:infinfhar} that \;$u\geq T_{\infty}(u)= v$. Thus, we get \;$T_{\infty}(u)=u$. Now if  \;$u=T_{\infty}(u)$, then using again lemma \ref{eq:infinfhar} or \eqref{c1}-\eqref{eq:definitionti}, we obtain \;$u$\; is an infinity superharmonic function. Hence the proof of the lemma is complete.

\end{proof}
\vspace{8pt}

To run our \;$p$-approximation scheme for existence, another crucial ingredient that we will need is an appropriate characterization of the limit of sequence of solution \;$w_p$\; of the \;$p$-obstacle problem \eqref{pobstacleproblem} with obstacle \;$\psi_p$\; under uniform convergence of  both \;$w_p$\; and \;$\psi_p$. Precisely, we will need the following lemma.
\begin{lemma}\label{lem:solnofpobstacleconvergetoinftyobstacel}
If \;$w_p$\; is a solution to the\; $p$-obstacle problem \eqref{pobstacleproblem} with obstacle \;$\psi_p$\; that is, $w_p$ satisfies
\begin{equation}\label{pobstacleproblem1}
\begin{cases}
-\Delta_p w_p \geq  0  & \quad \quad \textrm{in} \quad \Omega\\
-\Delta_p w_p \, (w_p-\psi_p)=0  & \quad \quad \textrm{in} \quad \Omega\\
w_p  \geq \psi_p &\quad \quad  \textrm{in} \quad \Omega
\end{cases}
\end{equation}
in the viscosity sense and  if also that \;$w_p \to u_{\infty}$\; and \;$\psi_p \to \psi_{\infty}$\; locally uniformly in \;$\overline{\Omega},$ then \;$u_{\infty}$\; is a solution in the viscosity sense of the following system
\begin{equation}\label{obstacleproblemforinftylaplacian1}
\begin{cases}
-\Delta_ \infty w_{\infty} \geq  0  &\quad \quad \textrm{in} \quad \Omega\\
-\Delta_ \infty w_{\infty} \, (w_{\infty}-\psi_{\infty})=0  &\quad \quad \textrm{in} \quad \Omega\\
w_{\infty}  \geq \psi_{\infty} &\quad \quad  \textrm{in} \quad \Omega.
\end{cases}
\end{equation}
\end{lemma}

\begin{proof}
First  of all, note that since \;$w_p \geq \psi_p$,\;  $-\Delta_p w_p \geq  0$\; in the viscosity sense in \;$\Omega$\; for every \;$p$, \;$w_p \to u_{\infty}$,\;and\;$\psi_p \to \psi_{\infty}$\; both locally uniformly in \;$\overline{\Omega}$, and \;$\overline{\Omega}$ is compact,  we have \;$w_{\infty}\geq \psi_{\infty}$\; and \;$-\Delta_ \infty w_{\infty} \geq  0$ in the viscosity sense in $\Omega.$ It thus remains to prove that $-\Delta_ \infty u_{\infty} \, (w_{\infty}-\psi_{\infty}) =0 \quad \textrm{in} \quad \Omega$\; which (because of \;$w_{\infty}\geq \psi_{\infty}$\; in $\Omega$) is equivalent to $-\Delta_ \infty u_{\infty} =0\, \quad \textrm{in} \quad \{w_{\infty} >\psi_{\infty}\}:=\{x\in \Omega:\;\; w_{\infty} (x)>\psi_{\infty}(x)\}$. Thus to conclude the proof, we are going to show \;$-\Delta_ \infty w_{\infty} =0\, \quad \textrm{in} \quad \{w_{\infty} >\psi_{\infty}\}$. To that end, fix $y\in  \{w_{\infty} >\psi_{\infty}\}.$ Then, by continuity there exists an open neighborhood \;$V$\; of \;$y$\; in $\Omega$  such that \;$\overline V$\; is a compact subset of $\Omega$, and  a small real number \;$\delta>0$ such that $w_{\infty} > \delta > \phi_{\infty}$ in $\overline V$. Thus, from \;$w_p \to w_{\infty}$, \;$\psi_p \to \psi_{\infty}$\; locally uniformly in \;$\overline{\Omega}$, and \;$\overline V$\; compact subset of \;$\Omega$, we infer that for sufficiently large \;$p$
\begin{equation}\label{eq:inside}
w_p > \delta > \psi_p \quad \textrm{in} \quad \overline V.
\end{equation}
On the other hand, since \;$w_p$\; is a solution to the \;$p$ obstacle problem \eqref{pobstacleproblem} with obstacle \;$\psi_p$,\; then clearly \;$-\Delta_p w_p = 0$\; in \;$\{w_p > \psi_p \}:=\{x\in \Omega:\;\; w_p(x) > \psi_p(x) \}$. Thus, \eqref{eq:inside} imply \;$-\Delta_p w_p = 0$\; in the sense of viscosity in \;$V$. Hence, recalling that \;$w_p \to w_{\infty}$\;\; locally uniformly in \;$\overline{\Omega}$\; and letting \;$p \to \infty,$ we obtain  
\[
-\Delta_ \infty w_{\infty} = 0\;\;\;\text{in the sense of viscosity in}\;\; \;V.
\] 
Thus, since \;$y\in V$\; is arbitrary in \;$\{w_{\infty}>\psi_{\infty}\}$, then we arrive to 
\[
-\Delta_ \infty w_{\infty} = 0\;\;\;\text{in the sense of viscosity in}\;\; \;\{w_{\infty}>\psi_{\infty}\},
\] 
thereby ending the proof of the lemma.
\end{proof}
\vspace{10pt}

On the other hand, to show the convergence of the minimal values of \;$J_p$\; to that of \;$J_{\infty}$, we will  make use of the following elementary results.
\begin{lemma}\label{eq:liminfmax}
Suppose \;$\{a_p\}$\; and \;$\{b_p\}$\; are nonnegative sequences with \;$$\liminf_{p \to \infty}a_p=a\; \;\;\text{and }\;\;\;\liminf_{p \to \infty} b_p = b.$$ Then
\[
\liminf_{p \to \infty} \max\{a_p, b_p\} = \max\{a, b\}.
\]
\end{lemma}
\begin{proof}
Let $\{b_{p_k}\}$ be a subsequence converging to $b = \liminf_{p \to \infty} b_p.$ Then $$\lim_{k \to \infty} \max\{a_{p_k}, b_{p_k}\} = \max \{a, b\}.$$ Since the $\liminf$ is the smallest limit point we have
\begin{equation}\label{i1}
\liminf_{p \to \infty} \max\{a_p, b_p\}  \leq \max\{a, b\}.
\end{equation}
On the other hand \;$$a_p, \, b_p \leq \max\{a_p, b_p\},\;\;\text{ for all}\;\;\; p.$$Thus $$b= \liminf_{p \to \infty} b_p \leq \liminf_{p \to \infty}  \max\{a_p, b_p\},$$ and likewise $$a \leq \liminf_{p \to \infty}  \max\{a_p, b_p\}.$$ Consequently
\begin{equation}\label{i2}
\liminf_{p \to \infty} \max\{a_p, b_p\}  \geq \max\{a, b\}.
\end{equation}
Finally \eqref{i1} and \eqref{i2} conclude the proof of the lemma .
\end{proof}

\begin{lemma}\label{limofsumofpowersofp}
Suppose \;$\{a_p\}$\; and \;$\{b_p\}$\; are nonnegative sequences with \;$$\liminf_{p \to \infty}a_p=a\; \;\;\text{and }\;\;\;\liminf_{p \to \infty} b_p = b.$$Then
\[
\liminf_{p \to \infty} (a_p^p + b_p^p)^{1/p} =\max\{a, b\}.
\]
\end{lemma}
\begin{proof}
It follows directly from the trivial inequality
\[
2^{\frac{1}{p}}\max\{a_p, b_p\}\geq (a_p^p + b_p^p)^{1/p}\geq\max\{a_p, b_p\}  , \;\;\forall p\geq 1,
\] 
lemma \ref{eq:liminfmax} and the fact that \;$\lim\inf_n(a_nb_n)=(\lim _n a_n)(\lim\inf_n b_n)$\; if \;$\lim_n a_n>0$.
\end{proof}

\section{Existence of optimal control for \:$J_{\infty}$\; and limit of \;$C_p$}
In this section, we show the existence of an optimal control for \;$J_{\infty}$\; and show that \;$C_p$\; converges to \;$C_{\infty}$\; as \;$p\to \infty$. We divide it in two subsections. In the first one we show existence of an optimal control for \;$J_{\infty}$\; via the \;$p$-approximation technique, and in the second one we show that \;$C_p$\; converges to \;$C_{\infty}$\; as \;$p$\; tends to infinity.
\subsection{Existence of optimal control}\label{limitofpproblem}
In this subsection, we show the existence of a minimizer of \;$J_{\infty}$\; via  the \;$p$-approximation technique using solutions of the optimal control for \;$J_p$. For this end, we start by recalling some optimality facts about \;$J_p$\; inherited from \;$\bar J_p$\;(see \eqref{eq:pp} for its definition)  and mentioned in the introduction. For  \;$\Omega \subset \mathbb R^n$\; a bounded and smooth domain, $z \in L^{\infty}(\Omega)$, $F \in Lip(\partial \Omega)$, and \;$1< p < \infty$, we recall that the functional \;$J_p$\; is defined by the formula
\begin{equation}\label{def}
J_p(\psi) =\left[ \int_{\Omega} |T_p(\psi) - z|^p + |D \psi|^p dx \right] ^{1/p},\;\;\;\psi\in W^{1, p}_{ F}(\Omega)
\end{equation}
and that the optimal control problem for $J_p$\; is the variational problem of minimizing \;$J_p$, namely
\begin{equation}\label{eq:pfnal}
\inf_{\psi\in W^{1, p}_{ F}(\Omega)}J_p(\psi) 
\end{equation}
over \;$W^{1, p}_{ F}(\Omega)$, where 
\[
W^{1, p}_{ F}(\Omega)= \{\psi \in W^{1,p} (\Omega) : \;\psi= F\, \;\;\;\; \textrm{on} \,  \;\;\partial \Omega\},
\]
and \;$T_p(\psi)$ is the solution to the \;$p$-obstacle problem with obstacle \;$\psi.$  Moreover, as for the functional \;$\bar J_p$,  \;$J_p$\; also admits a minimizer \;$\psi_p\in W^{1, p}_{ F}(\Omega)$\; verifying
 \begin{equation}\label{cspn}
T_p(\psi_p) = \psi_p.
\end{equation}
 As mentioned in the introduction, for more details about the latter results,  see \cite{ALY98} for \;$p =2$\; and see \cite{LH02} for \;$p>2$. 

\vspace{10pt}
To continue, let us pick \;$\eta \in W^{1, \infty}_F(\Omega)$. Since \;$\eta$\; competes in the minimization problem (\ref{eq:pfnal}), we have 
\[
\int_{\Omega} |D \psi_p|^p dx  \leq  J_p(\eta) = \int_{\Omega} |T_p(\eta) - z|^p + |D \eta|^p dx.
\]
Since \;$\overline \Omega$\; is compact and \;$T_p(\eta)\to T_{\infty}(\eta)$\; as \;$p\to \infty$\; locally uniformly on $\overline \Omega$ (which follows from the definition of \;$T_{\infty}(\eta)$),  we deduce that  for \;$p$\; very large
\begin{equation}\label{estimate of gradient of psi p}
\int_{\Omega} |D \psi_p|^p dx \leq M^p |\Omega|
\end{equation}
for some \;$M$\; which depends  only on \;$||\eta||_{W^{1, \infty}},$ $\;||T_{\infty}(\eta)||_{C^{0}}$\; and \;$||z||_{\infty}.$  Furthermore, let us fix \;$1<q < p.$ Then by using Holder's inequality, we can write
\begin{equation}\label{eq:holderonpsi_p}
\int_{\Omega}|D \psi_p|^qdx \leq \left\{\int_{\Omega}(|D \psi_p|^q)^{p/q} dx \right\}^{q/p} |\Omega|^{\frac{p-q}{p}}
\end{equation}
and we obtain by using (\ref{estimate of gradient of psi p}) that for \;$p$\; very large
\[
\int_{\Omega}|D \psi_p|^q dx\leq M^q |\Omega|^{\frac{q}{p}}|\Omega|^{\frac{p-q}{p}}
\]
and raising both sides to $1/q$, we derive that for \;$p$\; very large, there holds

\[
||D \psi_p||_{L^q}  \leq  M |\Omega|^{1/q},
\]
with \;$||\cdot ||_{L^q} $\; denoting the classical \;$L^q(\Omega)$-norm.
This shows, that the sequence \;$\{\psi_p\}$\; is bounded in \;$W^{1, q}_F(\Omega)$\; in the gradient norm for every \;$q$\; with a bound independent of \;$q$, and by Poincare's inequality, that for every \;$1<q<\infty$, the sequence \;$\{\psi_p\}$\; is bounded in \;$W^{1, q}_F(\Omega)$\; in the standard \;$W^{1,q}(\Omega)$-norm. Therefore , by classical weak compactness arguments, we have that, up to a subsequence, 
\begin{equation}\label{uniwe}
\psi_p\ \longrightarrow \psi_{\infty}, \;\; \text{as}\; \;p \to \infty\;\;\;\text{locally uniformly in}\;\;\; \overline\Omega\;\; \text{and weakly in} \;\;\;W^{1,q}(\Omega)\;\;\forall\ \;1<q<\infty.
\end{equation}
Notice that consequently $||D\psi_{\infty}||_{L^q} \leq M |\Omega|^{1/q}$ \;\;for all\;\; $1<q<\infty.$ Thus, we deduce once again by Poincare's inequality that
\begin{equation}\label{inspace}
\psi_{\infty} \in W^{1,\infty}_F(\Omega).
\end{equation}
\vspace{10pt}

We want now to show that \;$\psi_{\infty}$\; is a minimizer of \;$J_{\infty}.$ To that end, we make the following observation which is a consequence of lemma \ref{lem:solnofpobstacleconvergetoinftyobstacel}.
\begin{lemma}\label{lem:solnofTpconvergetoTinfty}
The function \;$\psi_{\infty}$\; is a fixed point of\; $T_{\infty}$, namely  $$ T_{\infty} (\psi_{\infty}) = \psi_{\infty},$$
and the solutions \;$T_p(\psi_p)$\; of the \;$p$-obstacle problem with obstacle \;$\psi_p$\; verify: as \;$p \to \infty$,  
$$
T_p(\psi_p)\ \longrightarrow T_{\infty}(\psi_{\infty})\;\;\;\text{locally uniformly in}\;\; \overline\Omega\;\; \text{and weakly in} \;\;W^{1,q}(\Omega)\;\;\forall\ \;1<q<\infty.
$$
\end{lemma}
\begin{proof}
We know that \;$T_p({\psi_p}) = \psi_p$ (see \eqref{cspn})  Thus using \eqref{uniwe} and Lemma \ref{lem:solnofpobstacleconvergetoinftyobstacel} with \;$\phi_p = \psi_p$\; and \;$w_p = T_p(\psi_p) = \psi_p$,  we have \;$T_p(\psi_p) \to \psi_{\infty}$\; locally uniformly in \;\;$\overline \Omega$, weakly in \;$W^{1,q}(\Omega)$\; for every \;$1<q<\infty$, and \;$\psi_{\infty}$\; is a infinity superharmonic.  Thus, recalling \eqref{inspace}, we have lemma \ref {eq:inftc} implies \;$T_{\infty} (\psi_{\infty}) = \psi_{\infty}$. Hence the proof of the lemma is complete.
\end{proof}
\vspace{10pt}

Now, with all the ingredients at hand, we are ready to show that \;$\psi_{\infty}$\; is a minimizer of \;$J_{\infty}.$  Indeed, we are going to show the following proposition:
\begin{proposition}
Let \;$\Omega \subseteq \mathbb R^n$\; be a bounded and smooth domain, $F\in Lip (\partial \Omega)$\; and \;$z \in L^{\infty}(\Omega).$  Then \;$\psi_{\infty}$\; is a minimizer of \;$J_{\infty}$\; on \;$W^{1,\infty}_F(\Omega)$ That is:
\[
J_{\infty}(\psi_{\infty})= \min_{\eta\in W^{1,\infty}_F(\Omega)}  J_{\infty}(\eta) 
\]
\end{proposition}

\begin{proof}
We first introduce for \;$n<p<\infty$\; and \;$\psi \in W^{1,p}_{F}(\Omega)$
\[
H_p(\psi) = \max\{||T_p(\psi) - z||_{\infty} , \; ||D \psi||_{\infty}\},
\]
which is well defined by Sobolev Embedding Theorem. Then for any \;$\eta  \in W^{1, \infty}_F(\Omega)$
\[
\int_{\Omega} |D \psi_p|^p dx  \leq  J_p^p(\eta) = \int_{\Omega} \left(|T_p(\eta) - z|^p + |D \eta|^p \right)dx.
\]
 Therefore, using the trivial identity \;$(|a|^p+|b|^p)^{\frac{1}{p}}\leq 2^{\frac{1}{p}}\max\{|a|,\;|b|\}$, we get
\[
\left( \int_{\Omega} |D \psi_p|^p dx \right)^{1/p} \leq 2^{1/p} |\Omega|^{1/p} H_p(\eta).
\]
If we now set 
\begin{equation}\label{eq:defnofip}
I_p = \inf_{\eta \in W^{1, \infty}_F(\Omega)} H_p(\eta),
\end{equation}
we deduce that 
\[
\left( \int_{\Omega} |D \psi_p|^p dx \right)^{1/p} \leq 2^{1/p} |\Omega|^{1/p} I_p.
\]
Let us fix \;$q$\; such that \;$n<q<\infty$. Then for \;$q<p<\infty$,\; by proceeding as in \eqref{eq:holderonpsi_p}, we obtain
\begin{equation*}
||D \psi_p||_{L^q}  \leq  2^{1/p} I_p |\Omega|^{1/q}.
\end{equation*}
Similarly, 
\[
||T_p(\psi_p) - z||_{L^q} \leq 2^{1/p} I_p |\Omega|^{1/q}.
\]
Thus
\begin{equation}\label{eq: boundofqnormofpsip}
\max \{||T_p(\psi_p) - z||_{L^q}, ||D \psi_p||_{L^q}\} \leq 2^{1/p} I_p |\Omega|^{1/q}.
\end{equation}
For any \;$\eta \in W^{1, \infty}_F(\Omega)$\; we also have \;$I_p  \leq H_p(\eta)$\; and \;$\liminf_{p \to \infty} I_p  \leq \liminf_{p \to \infty}  H_p(\eta).$  Thus, since \;$\psi_p$\; converges weakly in \;$W^{1,q}(\Omega)$\; to \;$\psi_{\infty}$\; as \;$p\to\infty$\; and \eqref{eq: boundofqnormofpsip} holds, then by weak lower semicontinuity, we conclude that
\[
||D\psi_{\infty}||_{L^q} \leq \liminf_{p \to \infty}  ||D\psi_p||_{L^q} \leq |\Omega|^{1/q} \liminf_{p \to \infty}  H_p(\eta).
\] 
Moreover, since \;$T_p(\eta)$\; converges locally uniformly on \;$\overline \Omega$\; to \;$T_{\infty}(\eta)$\; as \;$p\to\infty$\; and \;$\overline\Omega$\; is compact, then clearly
\[
\lim_{p \to \infty}  H_p(\eta) = J_{\infty}(\eta),
\]
and hence

\[
||D\psi_{\infty}||_{L^q} \leq J_{\infty}(\eta) |\Omega|^{1/q}.
\]
Since this holds for any element \;$\eta$\; of \;$W^{1, \infty}_F(\Omega)$,\; we conclude that by taking the infimum over \;$W^{1, \infty}_F(\Omega)$\; and letting \;$q \to \infty$
\begin{equation}\label{gradpsilessthaninfjinfty}
||D\psi_{\infty}||_{\infty} \leq \inf_{\eta \in W^{1, \infty}_F(\Omega)} J_{\infty}(\eta) \leq J_{\infty}(\psi_{\infty}).
\end{equation}
Using lemma \ref{lem:solnofTpconvergetoTinfty} and equation \eqref{eq: boundofqnormofpsip} combined with Rellich compactness Theorem or the continuous embedding of \;$L^{\infty}$\;into \;$L^q$, we conclude that
\[
||T_{\infty}(\psi_{\infty}) - z||_{L^q} = \lim_{p \to \infty}  ||T_p(\psi_p) - z||_{L^q} \leq |\Omega|^{1/q} \liminf_{p \to \infty}  H_p(\eta).
\]
Thus, as above letting \;$q$\; goes to infinity and taking infimum in \;$\eta$\; over \;$ W^{1, \infty}_F(\Omega)$,   we also have
\begin{equation}\label{Tinftypsilessthaninfjinfty}
||T_{\infty}(\psi_{\infty}) - z||_{\infty} \leq \inf_{\eta \in W^{1, \infty}_F(\Omega)} J_{\infty}(\eta) \leq J_{\infty}(\psi_{\infty}).
\end{equation}
Finally, from \eqref{inspace}, \eqref{gradpsilessthaninfjinfty} and \eqref{Tinftypsilessthaninfjinfty} we deduce
\[
J_{\infty}(\psi_{\infty}) = \min_{\eta \in W^{1, \infty}_F(\Omega)} J_{\infty}(\eta),
\]
as desired.
\end{proof}
\subsection{Convergence of Minimum Values}
In this subsection, we show the convergence of the minimal value of the  optimal control problem of \;$J_p$ \; to the one of \;$J_{\infty}$\; as \;$p\to \infty$, namely Theorem \ref{eq:convminima} via the following proposition:
\begin{proposition}
Let \;$\Omega \subset \mathbb R^n$ be a bounded and smooth domain, $F\in Lip (\partial \Omega)$\; and \;$1<p<\infty.$ Then recalling that 
\begin{equation*}
C_p= \min _{\psi \in  W^{1,p}_F(\Omega)}J_p(\psi)\; \;\;\text{and}\;\;\;\;C_{\infty}=\min _{\psi \in  W^{1, \infty}_F(\Omega)}J_{\infty}(\psi)  ,
\end{equation*}
we have 
\[
\lim_{p \to \infty} C_p= C_{\infty}.
\]
\end{proposition}
\begin{proof}
Let \;$\psi_p \in W^{1,p}_F(\Omega)$\; and \;$\psi_{\infty} \in W^{1,\infty}_F(\Omega)$\; be as in subsection \ref{limitofpproblem}. Then they satisfy \;$J_p(\psi_p) = C_p$\; and \;$J_{\infty}(\psi_{\infty}) = C_{\infty}$. Moreover, up to a subsequence, we have \;$\psi_p$\; and \;$\psi_{\infty}$\; verify \eqref{uniwe} and the conclusions of lemma \ref{lem:solnofTpconvergetoTinfty}. On the other hand, by minimality and H\"older's inequality, we have 
\begin{equation*}
J_p(\psi_p) \leq J_p(\psi_{\infty}) \leq 2^{1/p} |\Omega|^{1/p} \max \{||T_p(\psi_{\infty})-z||_{\infty}, ||D\psi_{\infty}||_{\infty} \}.
\end{equation*}
Thus
\begin{equation}\label{eq:limsup}
\limsup_{p \to \infty} J_p(\psi_p) \leq J_{\infty}(u_{\infty}).
\end{equation}
Now we are going to show the following 
\begin{equation}\label{minofinftyobstacleisless}
J_{\infty} (\psi_{\infty})  \leq \liminf_{p \to \infty} J_p(\psi_p).
\end{equation}
To that end observe that by definition of \;$J_{\infty}$, we have 
\begin{equation}\label{eq:apdef}
J_{\infty}(\psi_{\infty}) =\max \{ ||T_{\infty} (\psi_{\infty}) - z||_{\infty}, \, ||D\psi_{\infty}||_{\infty}\}.
\end{equation}
Thus, using the \;$L^q$-characterization of \;$L^{\infty}$, we have that \eqref{eq:apdef} imply
\begin{equation}\label{eq:aplq}
J_{\infty}(\psi_{\infty}) = \max \{\lim_{q \to \infty} ||T_{\infty} (\psi_{\infty}) - z||_{L^q}, \, \lim_{q \to \infty} ||D\psi_{\infty}||_{L^q}\},
\end{equation}
and by using lemma \ref{eq:liminfmax}, we get
\begin{equation}\label{eq:aplqs}
J_{\infty}(\psi_{\infty}) = \lim_{q \to \infty}\max \{ ||T_{\infty} (\psi_{\infty}) - z||_{L^q}, \,||D\psi_{\infty}||_{L^q}\}.
\end{equation}
On the other hand, by weak lower semicontinuity, and corollary \ref{lem:solnofTpconvergetoTinfty}, we have
\begin{equation}\label{eq:aplowsem}\
\,||D\psi_{\infty}||_{L^q}\leq \liminf_{p\rightarrow \infty} \,||D\psi_{p}||_{L^q}.
\end{equation}
Now, combining \eqref{eq:aplqs} and \eqref{eq:aplowsem}, we obtain
\begin{equation}\label{eq:aplqsn}
J_{\infty}(\psi_{\infty}) \leq \liminf_{q \to \infty}\max \{ ||T_{\infty} (\psi_{\infty}) - z||_{L^q},  \liminf_{p\rightarrow \infty} \,||D\psi_p||_{L^q}\}.
\end{equation}
Next, using lemma \ref{limofsumofpowersofp}, corollary \ref{lem:solnofTpconvergetoTinfty}, and \eqref{eq:aplqsn}, we get 
\begin{equation}\label{eq:aplqsnn}
J_{\infty}(\psi_{\infty}) \leq \liminf_{q \to \infty}\liminf_{p \to \infty} \left\{ (||T_{p} (\psi_{p}) - z||_{L^q})^p  +  (||D\psi_{p}||_{L^q})^p\right\}^{1/p}.
\end{equation}
To continue, we are going to estimate the right hand side of \eqref{eq:aplqsnn}. Indeed, using H\"older's inequality, we have
\begin{align*}
(||T_{p} (\psi_{p}) - z||_{L^q})^p  &= \left\{\int_{\Omega} |T_{p} (\psi_{p}) - z|^q dx \right\}^{p/q}\\
&\leq \left\{\int_{\Omega} |T_{p} (\psi_{p}) - z|^p dx \right\}  |\Omega| ^{(1-q/p)p/q} \\
&= \left\{\int_{\Omega} |T_{p} (\psi_{p}) - z|^p dx \right\}  |\Omega| ^{(1-q/p)p/q}.\\
\end{align*}
Similarly,  we obtain \;$$(||D\psi_{p}||_{L^q})^p \leq  \left\{ \int_{\Omega}  |D\psi_{p}|^p \, dx \right\}  |\Omega| ^{(1-q/p)p/q}.$$ By using the latter two estimates in \eqref{eq:aplqsnn}, we get
\begin{align*}
J_{\infty}(\psi_{\infty}) &\leq \liminf_{q \to \infty}\liminf_{p \to \infty} \left[\left\{ \int_{\Omega}( |T_{p} (\psi_{p}) - z|^p  +  |D\psi_{p}|^p )\, dx \right\}^{1/p} |\Omega| ^{(1-q/p)p/q(1/p)}\right]\\
&= \liminf_{q \to \infty}\liminf_{p \to \infty} \left[\left\{ \int_{\Omega} (|T_{p} (\psi_{p}) - z|^p  +  |D\psi_{p}|^p )\, dx \right\}^{1/p} |\Omega| ^{\frac{1}{q} - \frac{1}{p} }\right]\\
& = \liminf_{q \to \infty}\left[|\Omega| ^{\frac{1}{q}  }\liminf_{p \to \infty} J_p(\psi_p)\right] =\liminf_{p \to \infty} J_p(\psi_p) \numberthis 
\label{j-inftylessthanliminfj-p}
\end{align*}
proving claim \eqref{minofinftyobstacleisless}.
Combining \eqref{eq:limsup} with \eqref{j-inftylessthanliminfj-p} we obtain  
\[
\lim_{p \to \infty} J_p(\psi_p)=J_{\infty}(u_{\infty}),
\]
and recalling that we were working with a possible subsequence, then we have that up to a subsequence
\[
\lim_{p \to \infty} C_p = C_{\infty}.
\]
Hence, since the limit is independent of the subsequence, we have  
\[
\lim_{p \to \infty} C_p= C_{\infty}
\]
as required.
\end{proof}

\end{document}